\documentclass{article}
\usepackage[english,french]{babel}
\usepackage{amsrefs}
\usepackage{amssymb,amsmath,amsthm,amsfonts,epsfig,graphicx,psfrag,color}
\usepackage[utf8]{inputenc}
\usepackage{amsrefs}
\usepackage{hyperref}
\usepackage{wrapfig}
\hypersetup{
    colorlinks,
    citecolor=black,
    filecolor=blue,
    linkcolor=black,
    urlcolor=blue
}
\usepackage{tikz-cd}
\usepackage[normalem]{ulem}
\def\hyph{-\penalty0\hskip0pt\relax}
\theoremstyle{plain}
\newtheorem{thm}{Theorem}[section]
\newtheorem{cor}[thm]{Corollary}
\newtheorem{lem}[thm]{Lemma}
\newtheorem{prop}[thm]{Proposition}
\newtheorem{prob}[thm]{Problem}
\theoremstyle{definition}
\newtheorem{rmk}[thm]{Remark}
\newtheorem{ex}[thm]{Example}
\newtheorem{df}[thm]{Definition}
\theoremstyle{remark}

\newtheorem{con}[thm]{Construction}

\newtheorem{maintheorem}{Theorem}

\newcommand{\R}{\mathbb R}

\newcommand{\Z}{\mathbb Z}
\newcommand{\N}{\mathbb N}
\newcommand{\U} {\mathcal U}

\newcommand{\Sdos} {{\mathbb S}^2}
\newcommand{\Tdos} {{\mathbb T}^2}

\newcommand{\Suno}{\mathbb{S}^1}
\renewcommand{\epsilon}{\varepsilon}

\newcommand{\expc}{\xi}

\DeclareMathOperator{\diam}{diam}

\DeclareMathOperator{\interior}{int}
\DeclareMathOperator{\dist}{dist}

\begin{document}

\author{M. Achigar, A. Artigue and J. Vieitez\footnote{the authors are partially supported by Grupo de Investigaci\'on "Sistemas Din\'amicos" CSIC (Universidad de la
Rep\'ublica), SNI-ANII and PEDECIBA, Uruguay}}
\title{New cw\hyph expansive homeomorphisms of surfaces}
\selectlanguage{english}
\date{\today}
\maketitle

%

\selectlanguage{english}
\begin{abstract}
In this article we characterize monotone extensions of cw\hyph expansive homeomorphisms of compact metric spaces.
For this purpose we introduce the notion ``half cw\hyph expansivity" and
we study its natural quotient space, specially in the case of compact surfaces.
These results are applied to construct new examples of cw\hyph expansive homeomorphisms
of compact surfaces with infinitely many fixed points and empty wandering set.
These examples are quotients of topological perturbations of pseudo-Anosov diffeomorphisms.
We also show that there is a cw\hyph expansive homeomorphism with the shadowing property of the 2-sphere.
\end{abstract}

\section{Introduction}
In the theory of Dynamical Systems some topological properties play a key role.
Most of them are shared by hyperbolic systems as for instance expansivity,
specification and shadowing property.
With respect to expansivity, several generalizations have been proposed as for instance $N$\hyph expansivity ($N\geq 1$) \cite{Mo1}, measure\hyph expansivity \cite{MoSi}, countable\hyph expansivity \cite{ArCa}, hyper\hyph expansivity \cite{ArH},
\begin{wraptable}{r}{5.4cm}
\begin{tabular}{|c|}
  \hline\hline
  \phantom{\LARGE X}
  hyper\hyph expansivity
  \phantom{\LARGE X}\\
  $\downarrow$ \\
  expansivity $\leftrightarrow$ 1\hyph exp \\
  $\downarrow$ \\
  \hspace{1cm}$N$\hyph exp $\rightarrow$ $h$\hyph exp\\
  $\downarrow$ \\
  measure\hyph exp $\leftrightarrow$ countable\hyph exp  \\
  $\downarrow$ \\
  cw\hyph exp\\
  \hline\hline
\end{tabular}
\end{wraptable}
 $h$\hyph expansivity (entropy expansivity) \cite{Bowen72} and
cw\hyph expansivity (continuum\hyph wise expansivity) \cite{Kato93}.
See the diagram relating the classes of expansivity considered.

As we can observe, the most general properties in the diagram are cw\hyph expansivity and $h$\hyph expansivity. The relation between $h$\hyph expansivity and cw expansivity is not direct.
 Indeed, isometries are $h$\hyph expansive and not cw\hyph expansive, and a
 pseudo\hyph Anosov map of the 2-sphere (see \cite{PaVi}) is cw\hyph expansive but not entropy expansive (although it is asymptotically entropy expansive). 
But under generic conditions in the $C^1$ setting and $C^1$-far from homoclinic tangencies we have that measure\hyph expansivity, with a measure absolutely continuous with respect to Lebesgue measure, and $h$\hyph expansivity both hold \cites{DFPV,PaVi15}.
In this paper we address the study of cw\hyph expansivity defined on compact metric spaces that, as we have said above, is a property shared for several dynamical systems exhibiting chaotic behavior. For instance, Anosov diffeomorphisms, pseudo\hyph Anosov diffeomorphisms, pseudo\hyph Anosov maps with 1-prongs of the sphere $\Sdos$ and expansive homeomorphisms all are cw-expansive.

Among the results of this article we exhibit a new class of examples 
of cw-expansive homeomorphisms with infinitely many fixed points and without wandering points, we show that a generalized pseudo\hyph Anosov map of $\Sdos$ has the shadowing property and we generalize the notion of cw\hyph expansivity in a way that allows us to obtain an open property in the $C^0$-topology. We call this generalization \emph{half cw\hyph expansivity}.

Let us introduce some definitions to explain the results of this paper. Let $f\colon M\to M$ be a homeomorphism of a metric space $(M,\dist)$.
We say that $C\subseteq M$ is a \emph{continuum} if it is compact and connected.
Following Kato \cite{Kato93}, $f$ is said to be \emph{cw\hyph expansive} (continuum-wise expansive) if there is $\expc>0$ such that
if $C\subseteq M$ is a continuum and $\diam f^n(C)\leq\expc$ for all $n\in\Z$ then
$C$ is a singleton. In this case we say that $\expc$ is a \emph{cw\hyph expansivity constant}.
We recall that $f$ is \emph{expansive} if there is $\expc>0$ such that $\dist\bigl(f^n(x),f^n(y)\bigr)\leq\expc$ for all
$n\in\Z$ implies $x=y$.
It is remarkable that on compact surfaces expansive homeomorphisms are conjugate to
pseudo-Anosov diffeomorphisms \cites{L,Hi}.

In this paper we consider perturbations of cw\hyph expansive homeomorphisms in the $C^0$-topology.
In \cite{Le83} Lewowicz considered this kind of perturbations for expansive
homeomorphisms. He proved that if $f$ is sufficiently close to a given expansive homeomorphism of a compact metric space $M$, then there is $\expc>0$ with the following property:
$$\text{if }\dist\bigl(f^n(x),f^n(y)\bigr)\leq\expc\text{ for all }n\in\Z\text{ then }\dist(x,y)\leq\expc/2.$$
This property of $f$ allowed him to define the equivalence relation on $M$: $x\sim y$ if
$\dist\bigl(f^n(x),f^n(y)\bigr)\leq\expc$ for all $n\in\Z$.
In \cite{Le83} it is claimed that in this case the quotient space $\tilde M=M/\sim$ is metrizable
and the induced homeomorphism $\tilde f\colon \tilde M\to\tilde M$ is expansive.
The details of this construction were given in \cite{CS} by Cerminara and Sambarino.

If $f$ is a $C^0$-perturbation of a cw\hyph expansive homeomorphism then a similar situation arises.
In this case $f$ satisfies the following condition:
\begin{equation}
 \label{eqPertCwExp}
 \text{if }\diam f^n(C)\leq\expc\text{ for all }n\in\Z\text{ then }\diam C\leq\expc/2
\end{equation}
for every continuum $C\subseteq M$.
In \cite{FG} the techniques of \cite{CS} were applied to this case, proving that the corresponding quotient is cw\hyph expansive.

In the present paper we will consider this property, independently of any perturbation of a cw\hyph expansive homeomorphism. To this end we introduce the concept of half cw\hyph expansivity. 
For a metric space $(M,\dist)$, a homeomorphism $f\colon M\to M$ and $\expc>0$, a subset $C\subseteq M$ is called \emph{$\expc$-stable} if $\diam f^n(C)\leq\expc$ for all $n\in\Z$.

\begin{df} \label{mitad} Let $(M,\dist)$ be a metric space. A homeomorphism $f\colon M\to M$ is
\emph{half cw\hyph expansive} if there exists
$\expc>0$, such that every $\expc$-stable continuum is $\expc/2$-stable.
In this case we say that $\expc$ is
a \emph{half cw\hyph expansivity constant} and
that $f$ is \emph{half cw\hyph expansive relative to $\dist$ and $\expc$}.
\end{df}

We will prove that half cw-expansivity is an open property in the $C^0$\hyph topology (see Theorem \ref{prop:unifhalfcw}).
It is easy to show that neither expansive homeomorphisms nor cw-expansive ones are open in the $C^0$\hyph topology, so that the study of half cw-expansive homeomorphisms can be thought as an intent to remedy the lack of this property for expansive and cw-expansive homeomorphisms.
\begin{maintheorem}
\label{Teo A}
The set of half cw-expansive homeomorphisms of a compact metric space is open in the $C^0$\hyph topology.
Moreover, if $f\colon M\to M$ is a  half cw\hyph expansive homeomorphism of a compact metric space with constant $\expc$ then there is a $C^0$\hyph neighborhood $\U$ of $f$ such that every $g\in\U$ is half cw\hyph expansive with constant $\expc$.
\end{maintheorem}

For a half cw\hyph expansive homeomorphism $f\colon M\to M$ with constant $\expc$ it is natural to identify two points of $M$ if they lie in a common $\expc$-stable continuum, and indeed this identification defines an equivalence relation on $M$.
In \S \ref{subMonExt} we characterize monotone extensions of cw\hyph expansive homeomorphisms. That is, we consider
commutative diagrams of the form
\begin{equation}
\label{ecuDiag}
\begin{tikzcd}
  M \arrow{r}{f} \arrow{d}{q} & M \arrow{d}{q} \\
  N \arrow{r}{g}& N
\end{tikzcd}
\end{equation}
where $M,N$ are compact metric spaces, $f,g$ are homeomorphisms, $g$ is cw\hyph expansive and
$q$ is continuous, onto and monotone (i.e., the preimage set of any point is connected).
We show that if $f$ is half cw\hyph expansive then the quotient explained above gives a cw\hyph expansive map $g$ and a monotone canonical map $q$. Also, we prove the converse, for every monotone extension $f$ of a cw\hyph expansive homeomorphism $g$ there is a compatible metric in $M$ that makes $f$ a half cw\hyph expansive homeomorphism.

Next we consider cw\hyph expansivity on compact surfaces.
In \S \ref{secMetSp} we show that if $f$ is a monotone extension of a cw\hyph expansive homeomorphism $g$ as in diagram \eqref{ecuDiag}, the classes (preimages by $q$ of singletons) are sufficiently small and $M$ is a compact surface, then $N$ is homeomorphic to $M$, see Theorem \ref{thm:tildeM=M}.

\begin{maintheorem}\label{Teo B}
If $M$ is a closed surface with a Riemannian metric,
then there is $\epsilon_0>0$ such that if
$f\colon M\to M$ is a half cw\hyph expansive homeomorphism with half cw\hyph expansivity constant $\expc\leq\epsilon_0$
then the quotient space $\tilde M$ is homeomorphic to $M$.
\end{maintheorem}

This result is based on Moore's Theorem on plane decompositions \cite{Mo25} in the version of Roberts-Steenrod \cite{RoSt} for compact surfaces.
For manifolds of arbitrary dimension we show that no equivalence class induced by $q$
separates $M$, which implies that the codimension-one Betti number of each class is zero.

In \S\ref{secSurfWBdryCircle} we consider some examples on compact surfaces with boundary.
Also, we prove that for small enough constant of half cw-expansivity 
no non-trivial Peano space on $\R^2$ admits a half cw-expansive homeomorphism. In particular this is valid for the circle $\Suno$.

\begin{maintheorem} \label{Teo C}
The circle only admits trivial half cw\hyph expansive homeomorphisms. Moreover, if $X$ is a nontrivial Peano space contained in the plane then $X$ only admits trivial half cw-expansive homeomorphisms.
\end{maintheorem}

On compact surfaces there are cw\hyph expansive homeomorphisms that are not expansive.
Some of them are $m$-\emph{expansive}.
We recall that for $m\geq 1$ a homeomorphism $f$ is $m$-\emph{expansive} \cite{Mo1} if there is $\expc>0$ such that if $A\subseteq M$ and $\diam f^n(A)\leq\expc$ for all $n\in\Z$ then $A$ has at most $m$ points.
In \cite{APV} it is shown that the genus two surface admits a 2\hyph expansive homeomorphism that is not expansive.
In \cite{ArDend} it is shown that a pseudo-Anosov map with 1-prong singularities of the
2-sphere is cw\hyph expansive but not $m$\hyph expansive for any $m\geq1$. In \cite{ArRobN} the example of \cite{APV} is generalized and it is proved that there are $C^r$-robustly, $r\geq2$,
$m$\hyph expansive diffeomorphisms that are not Anosov diffeomorphisms.
In \cite{ArAnomalous} another variation is considered to prove that a local stable set may be connected but not locally connected, for a cw\hyph expansive homeomorphism of a surface.

All the examples mentioned in the previous paragraph have a finite number of fixed points.
In \S \ref{secPartPert} new examples of cw-expansive homeomorphisms are built. These examples, defined on the torus $\Tdos$, have the particular feature that have infinitely many fixed points and moreover their non-wandering set is the whole manifold.
\begin{maintheorem} \label{Teo D}
There exist cw\hyph expansive homeomorphisms on tori that has infinitely many fixed points and empty wandering set.
\end{maintheorem}

In particular these homeomorphisms are not $m$\hyph expansive for any $m\geq1$.
We recall that $x\in M$ is \emph{wandering} for $f\colon M\to M$ if there is an open set $U$ such that $U\cap f^n(U)=\varnothing$
for all $n\neq 0$.
For the construction of such examples we start with an Anosov diffeomorphism on the 2-torus.
Then we perform a suitable $C^0$-perturbation to obtain infinitely many fixed points.
Finally, we consider a quotient that gives a cw\hyph expansive homeomorphism, which we already know that is defined on a 2-torus again. The hard part, for our purposes, is to perform the perturbation in such a way that:
1) it adds no wandering point and
2) the (infinitely many) fixed points are not identified in the quotient, that is, there must be no continuum with small iterates containing any pair of the fixed points.
We will consider area preserving perturbations to ensure 1).

%
%

We finish in \S \ref{secCwSh} showing that there is a cw-expansive homeomorphism that is not hyperbolic and has the shadowing property. Indeed, we show a pseudo-Anosov homeomorphism with spines of $\Sdos$ that has this property (see Theorem \ref{thm:ejemplocw+sh}). 

\begin{maintheorem} \label{Teo E}
There exists a pseudo-Anosov homeomorphisms of $\Sdos$ that is cw\hyph expansive
and has the shadowing property.
\end{maintheorem}

For this proof we consider a homeomorphism $g\colon\Sdos\to\Sdos$ that is an \emph{antipodal quotient} of an Anosov automorphism of $\Tdos$ (see \S\ref{secCwSh}). 
It is well known that $g$ is not expansive. Then,
in light of Theorem \ref{Teo E}, we see that \cite{Lee}*{Corollary 2.4} may not be correct\footnote{In \cite{Lee}*{Corollary 2.4} it is claimed (among other things) that every cw\hyph expansive homeomorphism with the shadowing property on a compact manifold is expansive.}.
A direct proof of the non\hyph expansivity of $g$ can be found in \cite{Walters}*{Example 1, p. 140}.
Moreover, in \cite{ArDend}*{Proposition 2.2.2}
it is shown that for all $\epsilon>0$ there is a Cantor set $K\subseteq\Sdos$ such that $\diam g^n(K)\leq\epsilon$ for all $n\in\Z$.
Nevertheless, the most powerful argument to prove that $g$ is not expansive comes from \cites{Hi,L}: the 2-sphere does not admit expansive homeomorphisms.

\emph{Acknowledgements}. 
We wish to thank the referee for fruitful comments and suggestions on this paper.

\section{Extensions of cw\hyph expansive homeomorphisms}
\label{secExtCwExpHom}
Let $M$ be a topological space. A non-empty, compact and connected subset $C\subseteq M$ is called $\emph{continuum}$. A continuum is called \emph{trivial} if it has only one point.

\begin{df}[Kato \cite{Kato93}]
A homeomorphism $f\colon M\to M$ is \emph{cw\hyph expansive} if there exist a compatible metric $\dist$ on $M$ and $\expc>0$, such that every $\expc$-stable continuum is trivial. Such $\expc$ will be called a \emph{cw\hyph expansivity constant}.
\end{df}

If $\U$ is an open cover of $M$ and $C\subseteq M$,
we denote $C\prec\U$ to mean that $C\subseteq U$ for some $U\in\U$.
We say that $\U$ is a \emph{cw\hyph expansivity cover for $f$} if
$C\subseteq M$ is a continuum and $f^n(C)\prec\U$ for all $n\in\Z$ then $C$ is trivial.

\begin{lem}\label{lem:cwexptop<=>cwexpmet} If $M$ is a compact metrizable topological space and $f\colon M\to M$ is a homeomorphism then the following conditions are equivalent:
\begin{enumerate}
 \item The homeomorphism $f$ is cw\hyph expansive.
 \item There exists a cw\hyph expansivity cover $\U$ for $f$.
 \item For every compatible metric $\dist$ on $M$ there exists a cw\hyph expansivity constant $\expc$ for $f$.
\end{enumerate}
\end{lem}

The proof of Lemma \ref{lem:cwexptop<=>cwexpmet} is direct from the definitions.

\subsection{Half cw\hyph expansivity}
\label{subHalfCw}

Recall Definition \ref{mitad} of half cw-expansivity.


\begin{rmk} An equivalent condition for half cw\hyph expansivity is that for some $\expc>0$ every $\expc$-stable continuum has diameter less or equal than $\expc/2$.
\end{rmk}

\begin{rmk}
\label{rmkCteTrivial}
Notice that every homeomorphism of $M$ is half cw\hyph expansive if $\expc\geq2\diam M$.
In this case we say that $\expc$ is a \emph{trivial half cw\hyph expansivity constant}.
Naturally, the case of interest is for small half cw\hyph expansivity constants.
\end{rmk}

The next example shows that the existence of a non-trivial half cw\hyph expansivity constant depends on the compatible metric.
In \S \ref{secSurfWBdryCircle} we give examples of homeomorphisms of the 2-disk showing this phenomenon, see Propositions \ref{propHalfDepMet} and \ref{propHcwDisk}.

\begin{ex}Let $g\colon M\to M$ be a cw\hyph expansive homeomorphism of the compact metric space $(M,\dist')$
with cw\hyph expansivity constant $\expc=1$.
Consider $\dist(x,y)=\min\{\dist'(x,y),1\}$, a compatible metric on $M$.
Define $f\colon M\times [0,1]\to M\times [0,1]$ as $f(x,t)=(g(x),t)$.
On $M\times [0,1]$ consider the following two metrics:
\begin{equation*}
\begin{split}
\dist_1\bigl((x,t),(y,s)\bigr)&=\max\{\dist(x,y),|s-t|\},\\
\dist_2\bigl((x,t),(y,s)\bigr)&=\max\{4\dist(x,y),|s-t|\}.
\end{split}
\end{equation*}
Both metrics are compatible. For $\dist_1$ we have that no $\expc<2\diam_1M$ is a half cw\hyph expansivity constant,
but for $\dist_2$ we have the half cw\hyph expansivity constant $\expc=3<2\diam_2M$, where $\diam_i$ is the diameter associated to $\dist_i$.
\end{ex}

We recall that the $C^0$-topology in the set of homeomorphisms of a compact metric space $(M,\dist)$ is defined by the $C^0$-metric
\begin{equation}
\label{ecuDistC0}
\dist_{C^0}(f,g)=\sup_{x\in M}\dist\bigl(f(x),g(x)\bigr),
\end{equation}
for homeomorphisms $f,g\colon M\to M$.
We assume that the reader is familiar with the Hausdorff metric defined on the compact subsets of $M$.
We will use the fact that the space of subcontinua of $M$ is compact with this metric.
A proof can be found in \cite{Nadler}.

Theorem \ref{Teo A} follows immediately from  the following Theorem \ref{prop:unifhalfcw} item {\it 3.}
\begin{thm}\label{prop:unifhalfcw}
If $f\colon M\to M$ is a half cw\hyph expansive homeomorphism of a compact metric space with constant $\expc$ then:
\begin{enumerate}
\item There is $\alpha>\expc$ such that every $\alpha$-stable continuum $C\subseteq M$ is $\expc/2$-stable.
\item For every $\epsilon\in(\expc,\alpha)$ there is $m\in\N$ such that:
 $$\displaystyle\sup_{|n|\leq m}\diam f^n(C)\leq\alpha\text{ implies }\diam C<\epsilon/2$$
for every continuum $C\subseteq M$.
\item For every $\epsilon\in(\expc,\alpha)$ there is a $C^0$-neighborhood $\U$ of $f$ such that every $g\in\U$ is half cw\hyph expansive with constant $\epsilon$.
\end{enumerate}
\end{thm}

\begin{proof}
(1) Arguing by contradiction suppose that there is a sequence $(C_k)_{k\in\N}$
of $(\expc+1/k)$-stable continua that are not $\expc/2$-stable.
Since $\expc$ is a half cw\hyph expansivity constant, we have that $C_k$ is not $\expc$-stable for all $k\in\N$.
Then $\diam f^n(C_k)\leq\expc+1/k$ for all $n\in\Z$, and
for all $k\in\N$ there is $n_k\in\Z$ such that $\diam f^{n_k}(C_k)>\expc$.
By compactness we may assume that $\bigl(f^{n_k}(C_k)\bigr)_{k\in\N}$ converges to $C$ with respect to the Hausdorff metric.
Then $C$ is a continuum such that $\diam C\geq\expc$ and $\diam f^n(C)\leq\expc$ for all $n\in\Z$. This contradicts that $\expc$ is a half cw\hyph expansivity constant for $f$.

(2) To prove that the required $m\in\N$ exists we will argue by contradiction.
Suppose that there exists $\epsilon\in(\expc,\alpha)$ such that for each $m\in\N$ there is a continuum $C_m\subseteq M$ with $\sup_{|n|\leq m}\diam f^n(C_m)\leq\alpha$ and $\diam C_m\geq\epsilon/2$.
If $C\subseteq M$ is a limit continuum of $(C_m)_{m\in\N}$ then $\diam f^n(C)\leq\alpha$ for all $n\in\Z$ and $\diam C\geq\epsilon/2\geq\expc/2$. Then $C$ is an $\alpha$-stable continuum that is not $\expc/2$-stable, contradicting (1).

(3) For the given $\epsilon$ consider $m\in\N$ given by (2). Let $\U$ be a $C^0$-neighborhood of $f$ such that  $\sup_{|n|\leq m}\diam g^n(C)\leq\epsilon$ implies $\sup_{|n|\leq m}\diam f^n(C)\leq \alpha$ for every continuum $C\subseteq M$ and $g\in \U$. Let us show that $\epsilon$ is a half cw\hyph expansivity constant for every $g\in\U$. Suppose that $\diam g^n(C)\leq \epsilon$ for all $n\in\Z$. From the choice of $\U$, this implies that $\sup_{|n|\leq m}\diam f^n(C)\leq\alpha$. Then, as $m$ was chosen as in (2), we have that $\diam C<\epsilon/2$. That is, $\epsilon$ is a half cw\hyph expansivity
constant for every $g\in\U$.
\end{proof}

\begin{cor}
\label{cor:entornohalfcw} Let $f\colon M\to M$ be a cw\hyph expansive homeomorphism of a compact metric space with constant $\expc$. Then there is a $C^0$-neighborhood $\U$ of $f$ such that every $g\in\U$ is half cw\hyph expansive with constant $\expc$.
\end{cor}

\begin{proof}
 Suppose that $\expc$ is a cw\hyph expansivity constant of $f$.
 Notice that every $0<\epsilon\leq\expc$ is a half cw\hyph expansivity constant.
 By Theorem \ref{prop:unifhalfcw} we can take two half cw\hyph expansivity constants $\epsilon_1,\epsilon_2$ such that
 $\epsilon_1/2<\epsilon_2<\expc<\epsilon_1$.
 Let $\U_i$, $i=1,2$, be the neighborhoods given by Theorem \ref{prop:unifhalfcw}, such that
 if $g\in\U_i$ then $\epsilon_i$ is a half cw\hyph expansivity constant for $g$.
 Define $\U=\U_1\cap \U_2$. From our choice of $\epsilon_1$ and $\epsilon_2$ we see that $\expc$ is a
 half cw\hyph expansivity constant for every $g\in\U$.
\end{proof}

\subsection{Monotone extensions of cw\hyph expansive systems}
\label{subMonExt}

Let $M,N$ be topological spaces and $p\colon M\to N$ a map. We say that
$p$ is a \emph{quotient map} (or an \emph{extension map}) if $p$ is surjective and the quotient (final) topology of $N$ induced by $p$ and the topology of $M$ is the given topology of $N$.
If in addition $f$ and $g$ are homeomorphisms of $M$ and $N$, respectively, such that $p\circ f=g\circ p$, we say that $f$ is an \emph{extension of $g$ by $p$} (or that $g$ is a \emph{quotient of $f$ by $p$}), and we denote it by $(M,f)\stackrel{p}{\to}(N,g)$.

Let $M$ be a topological space and $f\colon M\to M$ a homeomorphism. An equivalence relation $\sim$ on $M$ is called \emph{compatible with $f$} if $x,y\in M$, $x\sim y$ implies $f(x)\sim f(y)$.
Given an equivalence relation $\sim$ on $M$, compatible with $f$, let $\tilde M$ be the topological quotient space, $q\colon M\to\tilde M$ the canonical map, and $\tilde f$ the homeomorphism of $\tilde M$ induced by $f$. Then $f$ is an extension of $\tilde f$ by $q$, and we say that this extension (or quotient) is \emph{induced by $\sim$}.

Any extension $(M,f)\stackrel{p}{\to}(N,g)$ is of the form $(M,f)\stackrel{q}{\to}(\tilde M,\tilde f)$, that is, there exists an equivalence relation $\sim$ on $M$ compatible with $f$ and a homeomorphism $h\colon N\to\tilde M$ such that $q=p\circ h$ and $\tilde f\circ h=h\circ g$. So we may suppose that extensions always comes from compatible equivalence relations.

A map between topological spaces is called \emph{monotone} \cite{Nadler} if the preimage set of any singleton of the codomain is connected.

\begin{rmk}\label{rmkMonCOnj}
If $M$ and $N$ are compact metric spaces and $q\colon M\to N$ is a monotone continuous and onto map, then
$q^{-1}(C)$ is connected for every connected subset $C\subseteq N$. See \cite{Nadler}*{Exercise 8.46}.
\end{rmk}

We will say that an extension $(M,f)\stackrel{q}{\to}(\tilde M,\tilde f)$ is a \emph{monotone extension} if the map $q$ is monotone, i.e., the equivalence classes $[x]$, $x\in M$, are connected.

\begin{df}\label{def:releqiv} Let $(M,\dist)$ be a metric space and $f\colon M\to M$ a half cw\hyph expansive homeomorphism with constant $\expc>0$. We consider the equivalence relation compatible with $f$ defined on $M$ by:
\begin{equation}\label{eqRelEquivCon}
x\sim y\quad\text{if}\quad x,y\in C\text{ for some }\expc\text{-stable continuum }C\subseteq M,
\end{equation}
for $x,y\in M$. Note that $\sim$ depends on $f$, $\dist$, and $\expc$. To simplify the terminology, in this context we refer to the extension $(M,f)\stackrel{q}{\to}(\tilde M,\tilde f)$
as the \emph{extension induced by $f$}.
Here $\tilde M$ denotes the quotient space $M/\sim$ and $q\colon M\to \tilde M$ is the canonical map.
\end{df}


\begin{lem}\label{lem:maxEstCont} Let $M$ be a compact metric space and $f\colon M\to M$ a half cw\hyph expansive homeomorphism with constant $\expc>0$. Then the equivalence classes of the relation $\sim$ of Definition \ref{def:releqiv} are the maximal $\expc$-stable continua.\footnote{That is, for all $x\in M$ the class $[x]$ is a $\expc$-stable continuum and, if $[x]\subseteq C$ for a $\expc$-stable continuum $C$ then $C=[x]$.} In particular, the canonical
map $q\colon M\to\tilde M$ associated to $\sim$ is monotone.
\end{lem}
\begin{proof}
Take $x\in M$ and consider the equivalence class $[x]\subseteq M$. Given $y\in[x]$
denote by $C_y$ a $\expc$-stable continuum containing $x$ and $y$.
Since $C_y\subseteq [x]$ for all $y\in[x]$, we have that $[x]=\bigcup_{y\in[x]}C_y$.
As each $C_y$ is connected and $x\in C_y$ for all $y\in [x]$, we conclude that $[x]$ is connected.
To prove that $[x]$ is closed, consider a sequence $(x_k)_{k\in\N}$ of elements of $[x]$ converging to a point $y\in M$. As $x\sim x_k$ for all $k\in\N$, there exists a sequence $(C_k)_{k\in\N}$ of $\expc$-stable continua such that $x,x_k\in C_k$ for all $k\in\N$. Taking a subsequence we can suppose that $(C_k)_{k\in\N}$ converges, in the Hausdorff metric, to a continuum $C$ that will be a $\expc$-stable continuum as well. Then, as $x,y\in C$, we have $y\in[x]$ and $[x]$ is closed.
Since $f$ is half cw-expansive, every $\expc$-stable continuum is $(\expc/2)$-stable.
Thus, given $y,z\in[x]$, there are $(\expc/2)$-stable continua $C_y,C_z$ such that $x,y\in C_y$ and $x,z\in C_z$.
This implies that $y,z\in C_y\cup C_z$ with $C_y\cup C_z$ a $\expc$-stable continuum. Consequently, $[x]$ is a $\expc$-stable continuum.
Finally, it is clear that if $C$ is a $\expc$-stable continuum and $[x]\subseteq C$ then $C=[x]$.
\end{proof}

Recall that a \emph{decomposition} of a space $M$ is a collection $\mathcal D$ of nonempty, mutually disjoint subsets of $M$ such that $\bigcup\mathcal D=M$. A decomposition $\mathcal D$ of a topological space $M$ is \emph{upper semi-continuous} if for every $D\in\mathcal D$ and every neighborhood $U$ of $D$ in $M$ there exists a neighborhood $V$ of $D$ in $M$ such that $D'\subseteq U$ for every $D'\in\mathcal D$ that meets $V$ (see for example \cite{Nadler}*{Definition 3.5}).

\begin{lem}\label{lem:cociente_metrizable} Let $M$ be a compact metric space, $f\colon M\to M$ a half cw\hyph expansive homeomorphism, $\sim$ the equivalence relation of Definition \ref{def:releqiv}, $\tilde M$ the quotient space by this relation and $q\colon M\to\tilde M$ the canonical map. Then:
\begin{enumerate}
\item[1.] The decomposition of $M$ into equivalence classes is upper semi\hyph continuous.
\item[2.] $q$ is a closed map.
\item[3.] The space $\tilde M$ is metrizable.
\item[4.] For every open set $U\subseteq M$ the set $\hat U=\{x\in M:[x]\subseteq U\}$ is open in $M$ and
$q(\hat U)$ is open in $\tilde M$.
\end{enumerate}
\end{lem}

\begin{proof}

(1) First observe that, as the classes are closed by Lemma \ref{lem:maxEstCont}, the upper semi-continuity of the decomposition of $M$ into equivalence classes amounts to: given $\epsilon>0$, and a convergent sequence $x_k\to x$ of $M$, there exists $k_0\in\N$ such that $[x_k]\subseteq_{\epsilon}[x]$ for all $k\geq k_0$. Here for subsets $A,B\subseteq M$, $A\subseteq_{\epsilon}B$ means that $d(a,B)<\epsilon$ for all $a\in A$, where $d(a,B)=\inf_{b\in B}d(a,b)$.

Arguing by contradiction suppose that this is not the case.
Then there exist an $\epsilon_0>0$, a convergent sequence $x_k\to x$ and a sequence $(y_k)_{k\in\N}$ such that $x_k\sim y_k$ and $d(y_k,[x])\geq\epsilon_0$ for all $k\in\N$.
Let $(C_k)_{k\in\N}$ be a sequence of $\expc$-stable continua such that $x_k,y_k\in C_k$ for all $k\in\N$.
As $M$ is compact, taking subsequences we may suppose that $(y_k)_{k\in\N}$ converges to an element $y$, which will satisfy $d(y,[x])\geq\epsilon_0$, and that $(C_k)_{k\in\N}$ converges in the Hausdorff metric to a continuum $C\subseteq M$ which will be $\expc$-stable. As $x\in C\cap[x]$ and both sets are in fact $\expc/2$-stable continua, we see that $C\cup[x]$ is a $\expc$-stable continuum and $x,y\in C\cup[x]$. Therefore $x\sim y$, or equivalently $y\in[x]$, which contradicts $d(y,[x])\geq\epsilon_0$.

(2) It follows form (1) and \cite{Nadler}*{Proposition 3.7}.

(3) It follows form (1) and \cite{Nadler}*{Theorem 3.9}.

(4) Note that $\hat U=M\setminus q^{-1}\bigl(q(M\setminus U)\bigr)$. Then, as $M\setminus U$ is closed and $q$ is a continuous and closed map we see that $\hat U$ is open. Finally $q(\hat U)$ is open because $q^{-1}\bigl(q(\hat U)\bigr)=\hat U$ is open.
\end{proof}

\begin{thm}\label{teo:cwexpgordo<=>extecwxp} Let $M$ be a compact metrizable space and $f\colon M\to M$ a monotone extension of $\tilde f\colon\tilde M\to\tilde M$ 
with quotient map $q\colon M\to\tilde M$.
The following statements are equivalent:
\begin{enumerate}
 \item $\tilde f$ is cw-expansive,
 \item $f$ is half cw-expansive with respect to a suitable compatible metric, with constant $\expc>0$, 
 and the extension induced by $f$ (as in Definition \ref{def:releqiv}) is precisely the given extension (that is,
 $q(x)=q(y)$ if and only if there is a $\expc$-stable continuum containing $x$ and $y$).
\end{enumerate}
%
\end{thm}

\begin{proof}
As we said before we may suppose that the extension comes from a suitable equivalence relation on $M$ compatible with $f$ which we call $\simeq$.
The class of $x\in M$ by $\simeq$ will be denoted as $[x]=q(x)$ where $q\colon M\to\tilde M$ is the canonical map.

($1\Rightarrow 2$). We need to show a compatible metric on $M$ and $\expc>0$ such that $f$ is half cw\hyph expansive relative to $\expc$, and such that the equivalence relation $\sim$ of Definition \ref{def:releqiv} coincides with $\simeq$.

Pick a compatible metric $\dist_1$ for $M$ and a compatible metric $\dist_2$ for $\tilde M$. As $\tilde f$ is cw\hyph expansive, by Lemma \ref{lem:cwexptop<=>cwexpmet} there exists a cw\hyph expansivity constant $\expc$ for $\tilde f$ with respect to $\dist_2$. Let $D_1$ be the diameter of $M$ with respect to $\dist_1$, $K=\expc/(1+2D_1)$ and define a new metric $\dist$ on $M$ by
$$\dist_3(x,y)=K\dist_1(x,y)+\dist_2([x],[y]),\qquad x,y\in M.$$
To prove that $\dist_3$ is compatible with $\dist_1$ first note that
$K\dist_1\leq\dist_3$.
On the other hand, by the uniform continuity of $q\colon (M,\dist_1)\to(\tilde M,\dist_2)$, it is easy to see that given any $\epsilon>0$ there exists $\delta>0$ such that $\dist_1(x,y)<\delta$ implies $\dist_3(x,y)<\epsilon$, for all $x,y\in M$.
Therefore $\dist_3$ is compatible.

We will show that $f$ is half cw\hyph expansive with respect to $\dist_3$ and $\expc$.
Note that, if $x,y\in M$ and $x\simeq y$ then
$$\dist_3(x,y)=K\dist_1(x,y)\leq KD_1\leq\expc/2.$$ Hence $\diam_3{}[x]\leq\expc/2$ for each $[x]\in\tilde M$, where $\diam_j$ stands for the diameter in the metric $\dist_j$.
Note also that as $\dist_2([x],[y])\leq\dist_3(x,y)$ for all $x,y\in M$ we have $\diam_2 q(C)\leq\diam_3 C$ for all $C\subseteq M$.
Suppose that $\diam_3 f^n(C)\leq\expc$ for all $n\in\Z$, hence $\diam_2 \tilde f^n\bigl(q(C)\bigr)=\diam_2 q\bigl(f^n(C)\bigr)\leq\diam_3 f^n(C)\leq\expc$ for all $n\in\Z$.
Since $\expc$ is a cw\hyph expansivity constant for $\tilde f$ we have that
$q(C)=[x]$ for some $x\in M$. Thus, $C\subseteq[x]$ and  $\diam_3 C\leq\diam_3{}[x]\leq\expc/2$.

Finally we prove that $\sim{=}\simeq$. If $x\sim y$ then there exists a $\expc$-stable continuum $C$ containing $x$ and $y$. In the previous paragraph we already showed that in this case $C$ is a subset of a single class (relative to $\simeq$), therefore $x\simeq y$. Conversely, take $x\simeq y$ and let $C=[x]$.
We know that $C$ is connected because the extension is assumed to be monotone. Then $x,y\in C$ and
$$\diam_3 f^n(C)=\diam_3[f^n(x)]\leq\expc/2$$ for all $n\in\Z$, hence $x\sim y$.

($2 \Rightarrow 1$). By Lemma \ref{lem:cociente_metrizable} we know that $\tilde M$ is a compact metrizable space.
To prove that $\tilde f$ is cw\hyph expansive, by Lemma \ref{lem:cwexptop<=>cwexpmet}, it suffices to show a cw\hyph expansivity cover for $\tilde f$.
We assume that the extension is induced by $\sim$,
the equivalence relation of Definition \ref{def:releqiv}.

Let $\dist$ be a compatible metric on $M$ and $\expc$ a half cw\hyph expansivity constant for $f$ with respect to $\dist$.
Consider $\alpha>\expc$, from Theorem \ref{prop:unifhalfcw},
such that every $\alpha$-stable continuum is $\expc/2$-stable.
For each $x\in M$ let $U(x)=B_{\alpha/2}(x)$ be the open ball of radius $\alpha/2$ centered at $x$, and consider
$$\hat U(x)=\{y\in M:[y]\subseteq U(x)\}.$$
By Lemma \ref{lem:cociente_metrizable} we know that $\hat U(x)$ is open, and, as $\diam{}[x] \leq \expc/2<\alpha/2$, we have that $[x]\subseteq U(x)$, so $x\in\hat U(x)$. Then, again by Lemma \ref{lem:cociente_metrizable}, $q\bigl(\hat U(x)\bigr)$ is an open neighborhood of $[x]$ in $\tilde M$ for all $x\in M$.
Consider the open cover of $\tilde M$ given by $\tilde\U=\bigl\{q\bigl(\hat U(x)\bigr):x\in M\bigr\}$.
To prove that $\tilde\U$ is a cw\hyph expansivity cover for $\tilde f$
suppose that $\tilde f^n(\tilde C)\subseteq q\bigl(\hat U(x_n)\bigr)$
for all $n\in\Z$ where $\tilde C\subseteq\tilde M$ is a continuum and $x_n\in M$.

Let $C=q^{-1}(\tilde C)$ and note that $f^n(C)\subseteq U(x_n)$ for all $n\in\Z$, because $q\bigl(f^n(C)\bigr)=\tilde f^n(\tilde C)\subseteq q\bigl(\hat U(x_n)\bigr)$ for all $n\in\Z$. Thus $C$ is a $\alpha$-stable set.
Since $q$ is monotone, by Remark \ref{rmkMonCOnj} we have that $C$ is connected, and therefore it must be $\expc$-stable. Hence, $C$ reduces to a single class, from which we conclude that $\tilde C$ is a trivial continuum. This proves that $\tilde\U$ is a cw\hyph expansivity cover for $\tilde f$, and therefore $\tilde f$ is cw\hyph expansive.
\end{proof}

\section{Monotone quotients on surfaces}
\label{secMetSp}

In this section we will consider the extensions of \S\ref{secExtCwExpHom}
for homeomorphisms of surfaces.
In \S\ref{secClosedSurf} we show that under certain conditions the quotient space of a closed surface is homeomorphic to the original surface.
In \S\ref{secSurfWBdryCircle} we consider the existence problem of half cw-expansive homeomorphisms on surfaces with boundary.

\subsection{Closed surfaces}
\label{secClosedSurf}

Here we prove Theorem \ref{Teo B} which is the main result of this subsection, see Theorem \ref{thm:tildeM=M}. We show that for a closed surface $M$ any half cw-expansive homeomorphism with a sufficiently small half cw-expansivity constant induces a quotient space $\tilde M$ homeomorphic to $M$. In order to prove Theorem \ref{thm:tildeM=M} we introduce Proposition \ref{prop:tildeM=M} which is a generalization, to arbitrary dimension, of what we need for surfaces.

In the sequel we denote as $b_r(C)$ the $r$-dimensional Betti number\footnote{i.e., the dimension of the $r^\text{th}$ homology group of $C$.} modulo 2
of the set $C$.

\begin{lem}\label{lem:betti0}
If $C\subseteq \R^n$ is a compact subset that does not separate $\R^n$ then $b_{n-1}(C)=0$.
\end{lem}
\begin{proof} Let $U=\R^n\setminus C$. Consider the $n$-sphere ${\mathbb S}^n$ as the one-point compactification of $\R^n$, ${\mathbb S}^n=\R^n\cup\{\infty\}$, and let $V=U\cup\{\infty\}$, which is an open and connected subset of ${\mathbb S}^n$ as can be easily seen. Then, as $b_{n-1}({\mathbb S}^n\setminus V)=0$, by \cite{Wilder}*{Theorem 5.25}, and $C=\R^n\setminus U={\mathbb S}^n\setminus V$, we conclude that $b_{n-1}(C)=0$.
\end{proof}

%

\begin{prop}\label{prop:tildeM=M}
Suppose that $M$ is a closed Riemannian manifold, $\dim M=d\geq 2$.
If $f\colon M\to M$ is a homeomorphism then there is
$\epsilon_0>0$ such that if $C\subseteq M$ is a maximal $\epsilon$-stable continuum with $0<\epsilon\leq\epsilon_0$
then $b_{d-1}(C)=0$.
\end{prop}

\begin{proof}
Let $\epsilon_0>0$ be a small constant so that if $A\subseteq M$ and $\diam A\leq\epsilon_0$
then there is a convex open ball $D\subseteq M$ containing $A$.
Suppose that $C\subseteq M$ is a maximal $\epsilon$-stable continuum
for some $0<\epsilon\leq\epsilon_0$.
For each $n\in\Z$ let $D_n$ be a convex open disc containing $f^n(C)$.
A connected component of $M\setminus f^n(C)$ will be called \emph{bounded component} if their closure is disjoint from the boundary of $D$.
It is clear that we can take $\epsilon_0>0$ sufficiently small so that $f$ preserves
the bounded components $M\setminus f^n(C)$.

To conclude that $b_{d-1}(C)=0$ we will apply Lemma \ref{lem:betti0} to $D$, which is homeomorphic to $\R^d$.
Arguing by contradiction suppose that $D\setminus C$ is disconnected.
Take $y\in D\setminus C$ in a bounded component.
Let $\gamma\subseteq D$ be a geodesic arc through $y$ with its extreme points in $\partial D$.
The point $y$ separates $\gamma$ in two arcs $\gamma_1$ and $\gamma_2$.
As $y$ is in a bounded component of $D\setminus C$ we can take $z_i\in\gamma_i\cap C$ for $i=1,2$.
Since $\diam C\leq\epsilon$ we have that $\dist(z_1,z_2)\leq\epsilon$.
This implies that $\dist(y,C)\leq\epsilon$ because $D$ is convex.
An analogous argument for each $f^n(C)\subseteq D_n$ shows that $\dist\bigl(f^n(y),f^n(C)\bigr)\leq\epsilon$ for all
$n\in\Z$.
Let $V$ be the \emph{unbounded} component of $D\setminus C$ and denote by $C'$
the continuum $D\setminus V$.
We have proved that $\diam f^n(C')\leq\epsilon_0$ for all $n\in\Z$.
The maximality of $C$ implies that $C'=C$,
but this contradicts that the class $C$ separates $D$.
\end{proof}



We recall the following result of Algebraic Topology, which will be used in Theorem \ref{thm:tildeM=M}.

\begin{thm}[\cite{RoSt}*{Theorem 1}]\label{rmkRoSt}
If $M$ is a compact connected surface without boundary
and $\tilde M$ is the quotient space induced by an upper-semicontinuous decomposition of $M$ into continua that contains at least two elements, and $b_1(C) =0$ for each equivalence class $C\subseteq M$, then $\tilde M$ is homeomorphic to $M$.
\end{thm}

\begin{thm}\label{thm:tildeM=M}
If $M$ is a closed surface with a Riemannian metric,
then there is $\epsilon_0>0$ such that if
$f\colon M\to M$ is a half cw\hyph expansive homeomorphism with half cw\hyph expansivity constant $\expc\leq\epsilon_0$
then $\tilde M$ is homeomorphic to $M$.
\end{thm}

\begin{proof}
Take $\epsilon_0$ from Proposition \ref{prop:tildeM=M}. Suppose that $f$ is half cw\hyph expansive with constant $\expc\leq\epsilon_0$.
By Lemma \ref{lem:maxEstCont} we know that the maximal $\expc$-stable continua are the equivalence classes that defines $\tilde M$.
Since $\dim M=2$ we have that $b_1(C)=0$ for every class $C\subseteq M$. 
Then the result follows by Theorem \ref{rmkRoSt}.
\end{proof}

For reference in \S \ref{secPartPert} we state the following direct consequence.

\begin{cor}\label{coroToro}
Let $M$ be a closed surface with a Riemannian metric and $f_0\colon M\to M$ be a cw\hyph expansive homeomorphism. Then there exists $\epsilon_1>0$ such that for every $0<\expc\leq\epsilon_1$ there is
a $C^0$-neighborhood $\U$ of $f_0$ such that every $f\in\U$ is half cw\hyph expansive with constant $\expc$ and
$\tilde M$ is homeomorphic to $M$.
\end{cor}

\begin{proof}
Take $\epsilon_0$ from Theorem \ref{thm:tildeM=M}, let $\expc_0$ be a cw\hyph expansivity constant for $f_0$ and define $\epsilon_1=\min\{\epsilon_0,\expc_0\}$. Then, for any $0<\expc\leq\epsilon_1$ we have that $\expc$ is a cw\hyph expansivity constant for $f_0$.  By Corollary \ref{cor:entornohalfcw} there exists a $C^0$-neighborhood $\U$ of $f_0$ such that all $f\in\U$ are half cw\hyph expansive with constant $\expc$. Finally, as $\expc\leq\epsilon_0$, we have that $\tilde M$ is homeomorphic to $M$ for every $f\in\U$.
\end{proof}

For the proof of the next result we recall some known facts.
Given two compact metric spaces $M,N$ and a continuous map $q\colon M\to N$, we say that $q$ is a \emph{near-homeomorphism}
\cite{Daverman}*{p. 27}
if every $C^0$-neighborhood of $q$ contains a homeomorphism from $M$ to $N$.

By Corollary \ref{cor:entornohalfcw} we know that in a $C^0$-neighborhood of a cw\hyph expansive homeomorphism
every homeomorphism is half cw\hyph expansive. The next result is some kind of converse for surfaces.
That is, in a suitable neighborhood of a half cw\hyph expansive homeomorphism there is a cw\hyph expansive homeomorphism.
The size of this neighborhood depends on the half cw\hyph expansivity constant.

\begin{thm}
If $f\colon M\to M$ is a half cw\hyph expansive homeomorphism
of a closed surface with a Riemannian metric, with constant $\expc\leq\epsilon_0$ (where $\epsilon_0$ is given by Theorem \ref{thm:tildeM=M}) and $\epsilon>0$ is given, then there is a cw\hyph expansive
homeomorphism $g\colon M\to M$ conjugate with $\tilde f$ such that $\dist_{C^0}(f,g)<\expc/2+\epsilon$.
\end{thm}

\begin{proof}
It is known that the quotient map $q\colon M\to\tilde M$ is a near-homeomorphism.
We sketch the proof for reader's convenience.
From Proposition \ref{prop:tildeM=M} and \cite{Kur}*{p. 514 Thm 6} we know that the equivalence classes are cell-like.
By \cite{Daverman}*{p. 187} a cell-like decomposition is shrinkable.
On compact metric spaces the shrinkability condition implies that the quotient map is a near-homeomorphism \cite{Daverman}.

Let $h_n\colon M\to\tilde M$ be a sequence of homeomorphisms converging to $q$ in the $C^0$-metric.
Define $g_n=h_n^{-1}\circ\tilde f\circ h_n$.
We know that $g_n$ is cw\hyph expansive because $\tilde f$ is cw\hyph expansive and they are conjugate.
Suppose that there are $\epsilon>0$ and $x_n$ such that
\begin{equation}
 \label{ecuCasiHomeo}
 \dist\bigl(f(x_n),h_n^{-1}\circ\tilde f\circ h_n(x_n)\bigr)\geq\expc/2+\epsilon
\end{equation}
for all $n\in\N$.
If $x$ is a limit point of $x_n$ then $\tilde f\circ h_n(x_n)\to \tilde f\circ q(x)$.
If $y$ is a limit point of $h_n^{-1}\circ\tilde f\circ q(x)$ then
$q(y)=\tilde f\bigl(q(x)\bigr)$.
Since $q\circ f=\tilde f\circ q$
we have that $q(y)=q(f(x))$.
We know that $\diam q^{-1}(z)\leq\expc/2$ for all $z\in\tilde M$.
Then, $\dist(y,f(x))\leq\expc/2$.
This contradicts \eqref{ecuCasiHomeo} and proves that for some $n$ it holds that $\dist_{C^0}(f,g_n)<\expc/2+\epsilon$.
\end{proof}

\subsection{Surfaces with boundary}
\label{secSurfWBdryCircle}
In this subsection we prove Theorem \ref{Teo C}.
We recall that surfaces with boundary do not admit cw\hyph expansive homeomorphisms.
This is a consequence of the non-existence of such homeomorphisms on the circle.
See for example \cite{ArDend}*{Remark 2.3.6}.
Also recall that $\expc>2\diam M$, is a trivial half cw\hyph expansive constant, see Remark \ref{rmkCteTrivial}.

\begin{prop}
\label{propNoCirc}
The circle only admits trivial half cw\hyph expansive homeomorphisms.
\end{prop}

\begin{proof}
Suppose that $f\colon M\to M$ is a half cw\hyph expansive homeomorphism of the circle $M=\Suno$ with constant $\expc$.
By Theorem \ref{teo:cwexpgordo<=>extecwxp} we have that $\tilde f\colon \tilde M\to\tilde M$ is cw\hyph expansive.
Since the canonical map is monotone we have that $\tilde M$ is a circle or a singleton.
As we said, it cannot be a circle. This implies that there is only one class, i.e., $f$ is trivially half cw\hyph expansive.
\end{proof}

The next two results show that the non-triviality of the
half cw\hyph expansivity depends on the compatible metric.

\begin{prop}
 \label{propHalfDepMet}
 Suppose that $D$ is homeomorphic to a 2-dimensional disk with a metric $\dist$ such that
 $\diam D=\diam\partial D$.
 Then $D$ admits no non-trivial half cw\hyph expansive homeomorphism.
\end{prop}

\begin{proof}
 If $\expc\leq2\diam D$ is a half cw\hyph expansivity constant for $f$
 then it is also a constant for $g=f|\partial D$.
 Then $g$ is a non-trivial half cw\hyph expansive homeomorphism of a circle, contradicting Proposition
 \ref{propNoCirc}.
\end{proof}

It is easy to see that for a disk embedded in the plane we have that
$\diam D=\diam\partial D$ with respect to the Euclidean metric.
Then, Proposition  \ref{propHalfDepMet} can be applied, for example, to the standard 2-disk $x^2+y^2\leq 1$ with the Euclidean metric.

\begin{prop}
\label{propHcwDisk}
 The closed 2-dimensional disk with a suitable metric admits a non-trivial half cw\hyph expansive homeomorphism.
\end{prop}

\begin{proof}
Let $\tilde f\colon \Sdos\to\Sdos$ be a cw\hyph expansive homeomorphism of the 2-sphere with hyperbolic fixed points.
For example we can take a power of the homeomorphism that will be explained in \S\ref{secCwSh}.
Suppose that $\Sdos=\R^2\cup\{\infty\}$, the origin is a hyperbolic fixed point and $\tilde f$ is linear in a neighborhood of $(0,0)$.
Let $D=\{v\in\Sdos:\|v\|\geq 1\}$ be a disk in the sphere, where $\|\cdot\|$ is the Euclidean norm and
$\infty\in D$.
Consider $q\colon D\to\Sdos$ given by
$q(r,\theta)=(r-1,\theta)$ in polar coordinates.
Note that $q(\partial D)=(0,0)$
and that $q$ is injective in the interior of $D$.
Define $f\colon D\to D$ as $f(x)=q^{-1}(\tilde f(q(x)))$ for all $x\in D\setminus\partial D$.
Since $\tilde f$ is linear around $(0,0)$ the map $f$ can be continuously defined in $\partial D$ obtaining a homeomorphism $f\colon D\to D$.
As $f$ is a monotone extension of $\tilde f$ and $\tilde f$ is cw\hyph expansive, by Theorem \ref{teo:cwexpgordo<=>extecwxp} we conclude that there is a compatible metric on $D$ that
makes $f$ a half cw\hyph expansive homeomorphism.
\end{proof}

Note that the example of Proposition \ref{propHcwDisk} has a non-trivial class, i.e., $b_1(\partial D)\neq 0$
where $\partial D$ is a class of the equivalence relation of Definition \ref{def:releqiv}.
The next result generalizes this remark for an arbitrary plane Peano continuum (i.e., a locally connected subcontinuum of $\R^2$).
It depends on \cite{Kato93}*{Theorem 6.2} where Kato proved that no non-trivial plane Peano continuum admits a cw\hyph expansive homeomorphism.

\begin{prop}
 If $M\subset\R^2$ is a Peano continuum and $f\colon M\to M$ is non-trivially half cw\hyph expansive
 then there is a class $[x]\subset M$ with $b_1([x])\neq 0$.
\end{prop}

\begin{proof}
Let $f\colon M\to M$ be a half cw\hyph expansive homeomorphism.
Suppose that $b_1([x])=0$ for all $x\in M$.
Let $\tilde M$ be the quotient space and $q\colon M\to\tilde M$ the canonical map.
By \cite{Nadler}*{Corollary 8.17} $\tilde M$ is a Peano continuum and by Theorem \ref{teo:cwexpgordo<=>extecwxp} $\tilde f$ is cw\hyph expansive on $\tilde M$.

We will show that $\tilde M$ is a plane Peano continuum.
Consider the decomposition $G$ of $\R^2$ given by $G(x)=\{x\}$ for $x\notin M$ and $G(x)=[x]$ for $x\in M$.
By \cite{Mo25} we have that $\R^2/G$ is homeomorphic to $\R^2$.
And given that $\tilde M\subset \R^2/G$ we conclude that $\tilde M$ is a plane Peano continuum.
Applying
\cite{Kato93}*{Theorem 6.2}
we have that $\tilde M$ is a singleton, that is, $f$ is trivially half cw\hyph expansive.
\end{proof}

With Proposition \ref{propHcwDisk} we can construct the following example that explains the meaning of $\epsilon_0$ in
Proposition \ref{prop:tildeM=M} and Theorem \ref{thm:tildeM=M}.

\begin{ex}
 Let $f_i\colon D_i\to D_i$, $i=1,2$, be two copies of the half cw\hyph expansive homeomorphism
 given in Proposition \ref{propHcwDisk}.
 Identifying the corresponding points of the boundaries of the disks we obtain
 a half cw\hyph expansive homeomorphism $f\colon \Sdos\to\Sdos$ of the 2-sphere.
 Let $\gamma$ be the circle in the sphere associated to the boundaries of the disks.
 In this case the quotient collapses the invariant circle $\gamma$ and the quotient space
 is not a surface, it is homeomorphic to the union of two tangent spheres in $\R^3$.
 Also, we see that $b_1(\gamma)\neq 0$.
\end{ex}


\section{Examples with infinitely many fixed points.}
In this section by means of a series of steps, named constructions, Theorem \ref{Teo D} is proved.
\label{secPartPert}
To this end we will perform a perturbation of a cw\hyph expansive homeomorphism of a compact surface in order to
obtain new examples of cw\hyph expansive homeomorphisms with particular properties.

Given a homeomorphism $f\colon M\to M$ and a closed set $D\subseteq M$, a \emph{modification of $f$ in $D$} is a homeomorphism $g\colon M\to M$ such that $f|_{M\setminus D}=g|_{M\setminus D}$. Every such modification $g$ is determined by a homeomorphism $h\colon D\to f(D)$ such that $h|_{\partial D}=f|_{\partial D}$ ($h=g|_D$). In this case we sometimes refer to $h$ itself as the modification.

\begin{rmk}
Note that if $g$ is a modification of $f$ in $D$ as before then $f(D)=g(D)$ and $\dist_{C^0}(f,g)\leq\diam f(D)$, where
$\dist_{C^0}$ was defined in \eqref{ecuDistC0}.
\end{rmk}

We start with the area preserving linear map $T\colon\R^2\to\R^2$ given by $T(x,y)=(\lambda x, \lambda^{-1}y)$ where $\lambda>1$ is fixed. This map transforms a line of equation $y=kx$ into the line $y=k\lambda^{-2}x$ and leaves invariant the hyperbolas $xy=k$.

In order to obtain examples without wandering points we will construct area preserving perturbations.
For this purpose we will need the following result, where $\mu$ stands for the Lebesgue measure.

\begin{thm}[\cite{OU}*{Corollary 3}]
\label{thmOU}
If $D,E\subseteq \R^2$ are diffeomorphic to closed rectangles,
$\mu(D)=\mu(E)$,
and $\partial S\colon\partial D\to\partial E$ is a homeomorphism then there is an area preserving homeomorphism $S\colon D\to E$ such that $S|_{\partial D}=\partial S$.
\end{thm}

The example is developed in a series of constructions.

\begin{con}
We start with a modification $T_0$ of $T$ as follows.
Consider the following subsets
\begin{equation*}
\begin{gathered}
D^+=\{(x,y):1\leq xy\leq2, \lambda^{-1}x\leq y\leq\lambda^3x\},\quad E^+=T(D^+),\\
D^-=\{(x,y):1/2\leq xy\leq1, \lambda^{-1}x\leq y\leq\lambda^3x\}, \quad E^-=T(D^-),\\
l_D=\{(x,y):xy=1, \lambda^{-1}x\leq y\leq\lambda^3x\},\quad l_E=T(l_D),\\
D=D^+\cup D^-,\quad E=T(D)=E^+\cup E^-.
\end{gathered}
\end{equation*}
Let $p$ and $q$ be the endpoints of the arc $l_D$ as in Figure \ref{figPert}.

\begin{figure}[h]
\centering
\includegraphics[scale=.583]{fig1.PNG}
\includegraphics[scale=.583]{fig2.PNG}
\caption{Construction of $T_0$}
\label{figPert}
\end{figure}
Note that $u=(1,1)\in l_D\cap l_E$.
Let $h\colon l_D\to l_E$ be a homeomorphism such that $h(p)=T(p)$, $h(q)=T(q)$ and $h(u)=u$.
Now consider the map $\partial T^+\colon\partial D^+\to\partial E^+$
given by $\partial T^+|_{\partial D^+\setminus l_D}=T|_{\partial D^+\setminus l_D}$ and $\partial T^+|_{l_D}=h$.
As $\partial T^+$ is a homeomorphism,
$D^+$ and $E^+$ are diffeomorphic to rectangles
and $\mu(D^+)=\mu(E^+)$, by Theorem \ref{thmOU} we can extend $\partial T^+$ to an area preserving homeomorphism $T^+\colon D^+\to E^+$. Analogously, we can find an area preserving homeomorphism $T^-\colon D^-\to E^-$ such that $T^-|_{\partial D^-\setminus l_D}=T|_{\partial D^-\setminus l_D}$ and $T^-|_{l_D}=h$. As $T^+$ and $T^-$ coincide (with $h$) in $l_D$ we have an area preserving homeomorphism $T_0\colon D\to E$ given by $T_0|_{D^+}=T^+$ and $T_0|_{D^-}=T^-$.
The map $T_0$ has $u\in\interior D$ as a fixed point. Besides, as $T_0|_{\partial D}=T|_{\partial D}$ we can define a modification of $T$ in $D$, replacing $T$ by $T_0$ on $D$. This modification will have $u$ as a fixed point and will be area preserving because $T_0$ and $T$ are.
\end{con}

\begin{con} We will define a family of modifications $T_n$, $n\in\N$.
For $n\in\N$ let $M_n\colon\R^2\to\R^2$ the homothetic transformation $M_n(v)=v/2^n$, $D_n=M_n(D)$, $E_n=M_n(E)$ and $T_n=M_n\circ T_0\circ M_n^{-1}$.
Note that $M_n$ leaves invariant the lines $y=kx$ and takes a hyperbola $xy=k$ to $xy=k/4^n$. It can be easily checked that $T_n\colon D_n\to E_n$ is an area preserving homeomorphism with a fixed point $u_n=(2^{-n},2^{-n})\in\interior D_n$, that $E_n=T(D_n)$ and that $T_n|_{\partial D_n}=T|_{\partial D_n}$. Then each $T_n$ gives a modification of $T$.
\end{con}

Notice that if $n\neq m$ in the previous construction then $\interior E_n\cap\interior E_m=\varnothing$, so that we can make the modifications $T_n$ simultaneously. In fact we want to perform all the modifications $T_n$ for $n\geq n_0$ simultaneously, with $n_0\in\N$ to be chosen later.

\begin{con}\label{con:tildeT_1} Given $n_0\in\N$ and $n\geq n_0$, define $\tilde T_n$ as $T$ with the modifications $T_{n_0},\ldots,T_n$, and define $\tilde T$ as $T$ with all the modifications $T_n$ for $n\geq n_0$.
\end{con}

Clearly all the maps $\tilde T_n$ of the previous construction are area preserving homeomorphisms.

\begin{lem} The map $\tilde T$ is an area preserving homeomorphism.
\end{lem}
\begin{proof}
First note that $\tilde T_n$ converges to $\tilde T$ pointwise, and that $\tilde T$ is bijective. For all $n>n_0$ we have
  $$\dist_{C^0}(\tilde T_{n-1},\tilde T_n)=\dist_{C^0}(T|_{D_n},T_n)\leq\diam E_n =2^{-n}\diam E_0,$$
from which we conclude that $\tilde T_n$ converges uniformly to $\tilde T$, and $\tilde T$ is continuous.
A similar argument applied to the inverses of all these maps shows that in fact $\tilde T$ is a homeomorphism.
Finally, for any measurable subset $A\subseteq\bigcup_{n\geq n_0}D_n$ we have $\tilde T(A)\subseteq\bigcup_{n\geq n_0}E_n$,
then
\begin{equation*}
\begin{split}
\mu\bigl(\tilde T(A)\bigr)&=\sum_{n\geq n_0}\mu\bigl(\tilde T(A)\cap E_n\bigr)=\sum_{n\geq n_0}\mu\bigl(\tilde T(A\cap D_n)\bigr)\\
&=\sum_{n\geq n_0}\mu\bigl(T_n(A\cap D_n)\bigr)=\sum_{n\geq n_0}\mu(A\cap D_n)=\mu(A).
\end{split}
\end{equation*}
As $\tilde T$ equals $T$ outside $\bigcup_{n\geq n_0}D_n$, we see that $\tilde T$ is area preserving.
\end{proof}

Clearly $\tilde T$ has infinitely many fixed points, at least the $u_n$, $n\geq n_0$. Besides, as the hyperbolas $H_n\colon xy=2/4^n$ do not meet $\interior D_m$ for all $n,m\in\N$, we see that $\tilde T=T$ on these hyperbolas.

\begin{con}\label{con:tildeT_2}
Given an open neighborhood $V$ of the origin in $\R^2$,
let $\expc>0$ be small enough such that if $K=[-\expc,\expc]^2$ then $K\subseteq V$ and $T(K)\subseteq V$. Let $L=[-\xi/2,\expc/2]^2$ and suppose that another neighborhood $W\subseteq L$ of the origin is given. Let $n_0\in\N$ be such that $D_n\cup E_n\subseteq W$ for all $n\geq n_0$, and perform the modification $\tilde T$ of $T$ of Construction \ref{con:tildeT_1}.
\end{con}

\begin{rmk}\label{rmk:modpert} Note that $\dist_{C^0}(\tilde T,T)\leq\diam W$, because $E_n\subseteq W$ for all $n\geq n_0$.
\end{rmk}

\begin{lem}\label{lem:Ttildenopegafijos} For any continuum $C$ of $\diam C\leq\expc/2$ containing two different fixed points $u_n$ and $u_m$ ($m>n\geq n_0$), there exist $N\in\N$ such that $\tilde T^k(C)\subseteq K\cup T(K)$ for $k=0,\ldots, N$ and $\diam \tilde T^N(C)>\expc/2$.
\end{lem}

\begin{proof} In fact, as $\diam C\leq\expc/2$, and $u_n\in L$ we see that $C\subseteq K$.
Now consider the branch $H^+$ on the first quadrant of the hyperbola $H_m\colon xy=2/4^m$.
As we can see $H^+$ separates $K$ in two components each of which containing one of the fixed points considered.
Therefore, as $C$ is a connected subset of $K$ containing both fixed points, we conclude that there exists a point $w\in H^+\cap C$. Now, as $T=\tilde T$ on $H^+$ we see that there exists $N'\in\N$ such that $\tilde T^{N'}(w)\notin K$. Consequently there exist a first $N\in\N$ such that $\tilde T^N(C)\not\subseteq K$. Clearly for this $N$ we have $\tilde T^k(C)\subseteq K\cup\tilde T(K)= K\cup T(K)$ for $k=0,\ldots, N$. Finally, as $\tilde T^N(C)$ is a closed set that meets $L$ (in $u_n$) and $\R^2\setminus K$ we conclude that $\diam \tilde T^N(C)>\expc/2$, because $\dist(L,\R^2\setminus K)=\expc/2$.
\end{proof}

Now we are ready to construct the desired example of a cw\hyph expansive homeomorphism of a compact surface admitting infinitely many fixed points and with no wandering points.

\begin{con}\label{con:elejemplo} Let $M=\R^2/\Z^2$ be the flat torus, and $f\colon M\to M$ the linear Anosov diffeomorphism given by the matrix $\bigl[\begin{smallmatrix}2&1\\1&1\end{smallmatrix}\bigr]$.
Let $U$ be an open neighborhood in $M$ of the fixed point of $f$ and assume that there is an isometric and area preserving local chart $\varphi\colon V\to U$, where $V$ is an open neighborhood of the origin in $\R^2$. Let $\lambda>1$ and $\lambda^{-1}$ the eigenvalues of $f$. We can also require that $f\circ\varphi(x)=\varphi\circ T(x)$, for $x\in T^{-1}(V)\cap V$, where $T$ is the linear map $T(x,y)=(\lambda x, \lambda^{-1}y)$ considered at the beginning of this section. As $\varphi$ is isometric, any modification $\tilde T$ of $T$ on a closed subset of $T^{-1}(V)\cap V$ gives a modification $g$ of $f$ such that, $\dist_{C^0}(f,g)=\dist_{C^0}(T,\tilde T)$.

As $f$ is an expansive homeomorphism, it is in particular cw\hyph expansive.
Let $\expc>0$ be an expansivity constant small enough to have $K\cup T(K)\subseteq V$ and $\expc\leq\epsilon_1$, where $K$ is as in Construction \ref{con:tildeT_2} and $\epsilon_1$ is from Corollary \ref{coroToro}.
Let $\delta>0$ be such that $B_{\delta}(f)\subseteq\U$ for the neighborhood $\U$ of Corollary \ref{coroToro}, and such that $W=B_{\delta/2}\bigl((0,0)\bigr)\subseteq L$, where $L$ is as in Construction \ref{con:tildeT_2}. With the $V$, $\expc$ and $W$ chosen, perform the perturbation $\tilde T$ of $T$ of Construction \ref{con:tildeT_2}. As this perturbation is in the closed subset $\overline W\subseteq T^{-1}(V)\cap V$, we have a corresponding perturbation $g$ of $f$ of the same size
$$\dist_{C^0}(f,g)=\dist_{C^0}(T,\tilde T)\leq\diam W<\delta,$$
where we used Remark \ref{rmk:modpert}. Then by the choice of $\delta>0$ we have that $g\in\U$, so that $g$ is half cw\hyph expansive with constant $\expc$. Consider the equivalence relation $\sim$ of Definition \ref{def:releqiv} associated to $g$ and $\expc$, and the homeomorphism $\tilde g$ on the quotient space $\tilde M$. By Theorem \ref{teo:cwexpgordo<=>extecwxp} we know that $\tilde g$ is cw\hyph expansive, and by Corollary \ref{coroToro} we have that $\tilde M$ is homeomorphic to $M$, a 2-torus.
\end{con}

\begin{thm}
\label{teoEjInfFij}
The cw\hyph expansive homeomorphism $\tilde g$ of the 2-torus obtained in Construction \ref{con:elejemplo} has infinitely many fixed points and empty wandering set.
\end{thm}
\begin{proof}
On one hand, as $\tilde T\colon K\to T(K)$ and $\varphi$ are area preserving, we see that $g$ is area preserving.
Then the wandering set of $g$ is empty. Consequently the wandering set of the quotient $\tilde g$ is empty. On the other hand, for $n\geq n_0$, where $n_0$ is as in Construction \ref{con:tildeT_2}, consider the fixed points $p_n$ of $g$ corresponding to the fixed points $u_n$ of $\tilde T$.
We will show that different fixed points $p_n$ and $p_m$ of $g$ are not identified by $\sim$, so that all this infinitely many points remains as infinitely many fixed points of $\tilde g$. In fact suppose on the contrary that $p_n\sim p_m$ with $m>n\geq n_0$. Then, by the  definition of $\sim$, there exists a continuum $C$ in $M$ containing $p_n$ and $p_m$, such that $\diam g^k(C)\leq\expc/2$ for all $k\in\Z$.
Then the continuum $C'=\varphi^{-1}(C)$ will satisfy $u_n,u_m\in C'$ and $\diam \tilde T^k(C')\leq \expc/2$, for all $k\in\Z$, which contradicts Lemma \ref{lem:Ttildenopegafijos}.
\end{proof}

\begin{rmk}
 It is clear that with the techniques developed in this section we can perturb an arbitrary pseudo-Anosov map of an arbitrary compact surface in a neighborhood of a periodic orbit.
\end{rmk}

\section{Cw\hyph expansivity with the shadowing property}
\label{secCwSh}

In this section we will prove Theorem \ref{Teo E}, i.e., that the 2-sphere admits a cw\hyph expansive homeomorphism with the shadowing property.
For the proof we develop some general results. Some of them could be well known, but as we have not found them in the literature and the proofs are short we include the details.

\begin{lem}\label{lem:uniformopenness} Let $M$ and $N$ be compact metric spaces and $q\colon M\to N$ a continuous and open map. Then for every $\rho>0$ there exists $\nu>0$ such that $q\bigl(B_{\rho}(x)\bigr)\supseteq B_{\nu}\bigl(q(x)\bigr)$ for all $x\in M$.
\end{lem}

\begin{proof} If this is not the case then there exist $\rho>0$ and a sequence $(x_n)_{n\in\N}$ in $M$ such that $q\bigl(B_{\rho}(x_n)\bigr)\not\supseteq B_{1/n}\bigl(q(x_n)\bigr)$ for all $n\in\N$. As $M$ is compact we may assume that $x_n\to x\in M$. Then $q(x_n)\to q(x)$ by the continuity of $q$.
Let $U=q\bigl(B_{\rho/2}(x)\bigr)$ which is open in $N$ because $q$ is an open map. Thus $U$ contains a neighborhood $B_{\delta}(q(x))$ for some $\delta>0$. Let $n_0$ be such that $x_n\in B_{\rho/4}(x)$ and $q(x_n)\in B_{\delta/2}\bigl(q(x)\bigr)$ for all $n\geq n_0$. Then $B_{\rho}(x_n)\supseteq B_{\rho/2}(x)$ and hence $q\bigl(B_{\rho}(x_n)\bigr)\supseteq B_{\delta/4}\bigl(q(x_n)\bigr)$ for every $n\geq n_0$, contradicting that $q\bigl(B_{\rho}(x_n)\bigr)\not\supseteq B_{1/n}\bigl(q(x_n)\bigr)$ for $n$ large enough.
\end{proof}

Let $f\colon M\to M$ be a homeomorphism on a metric space $(M,\dist)$.
Given $\delta>0$, a bi-infinite sequence $(x_n)_{n\in\Z}$ in $M$ is a \emph{$\delta$-pseudo orbit} if $\dist\bigl(f(x_n),y_{n+1}\bigr)<\delta$ for all $n\in\Z$.
If $(x_n)_{n\in\Z}$ is a $\delta$-pseudo orbit and $\epsilon>0$ we say that $x\in M$ \emph{$\epsilon$-shadows} the pseudo orbit if $\dist(f^n(x),x_n)<\epsilon$ for all $n\in\Z$.
We say that the homeomorphism $f\colon M\to M$ has the \emph{shadowing property} if
for all $\epsilon>0$ there is $\delta>0$ such that every $\delta$-pseudo orbit can be $\epsilon$-shadowed, and in this case we say that $\delta$ is \emph{shadowing constant} associated to $\epsilon$.

\begin{prop}\label{prop:semiconjsombreado}
Let $M$ and $N$ be compact metric spaces, $f\colon M\to M$ and $g\colon N\to N$ homeomorphisms and $q\colon M\to N$ a continuous and open onto map such that $q\circ f=g\circ q$.
If $f$ has the shadowing property then $g$ has the shadowing property.
\end{prop}

\begin{proof}
Given $\epsilon>0$ let $\epsilon'>0$ be such that for all $x,y\in M$, $\dist(x,y)<\epsilon'$ implies $\dist\bigl(q(x),q(y)\bigr)<\epsilon$. Let $\delta'>0$ be a shadowing constant associated to $\epsilon'$ for $f$, and apply Lemma \ref{lem:uniformopenness} to get $\delta>0$ such that $q\bigl({\delta'}(x)\bigr)\supseteq B_{\delta}\bigl(q(x)\bigr)$ for all $x\in M$. We claim that $\delta$ is a shadowing constant associated to $\epsilon$ for $g$.

Indeed, given a $\delta$-pseudo orbit $(y_n)_{n\in\Z}$ of $g$ we lift it to a $\delta'$-pseudo orbit of $f$ as follows.
Take $x_0\in M$ such that $q(x_0)=y_0$.
As $\dist\bigl(g(y_0),y_1\bigr)<\delta$ and $q\bigl(f(x_0)\bigr)=g(y_0)$, by the choice of $\delta$ we can find $x_1\in B_{\delta'}\bigl(f(x_0)\bigr)$ such that $q(x_1)=y_1$. Doing this inductively we see that we can lift the positive $\delta$-pseudo orbit $(y_n)_{n\geq0}$ to a positive a $\delta'$-pseudo orbit. A similar argument permits us to lift the negative $\delta$-pseudo orbit $(y_n)_{n\leq0}$, and so the entire $\delta$-pseudo orbit.
By the choice of $\delta'$ there exists $x\in M$ that $\epsilon'$-shadows the lifted $\delta'$-pseudo orbit.
Finally, by the choice of $\epsilon'$ we conclude that $y=q(x)$ $\epsilon$-shadows the initially given $\delta$-pseudo orbit.
\end{proof}

\begin{lem}\label{lem:hayconexo}
Let $M$ and $N$ be compact metric spaces and $q\colon M\to N$ a continuous and open map. Then for every non trivial continuum $C\subseteq N$ there exists a non trivial subcontinuum of $q^{-1}(C)$.
\end{lem}

\begin{proof} Suppose on the contrary that $D=q^{-1}(C)$ is a totally disconnected set. Take $x,y\in D$ such that $q(x)\neq q(y)$. By continuity of $q$ there exists a neighborhood
$V\subseteq D$ of $x$ relative to $D$ such that $q(y)\notin q(V)$.
As $D$ is totally disconnected we may assume that $V$ is open and closed relative to $D$. Then $V=V_1\cap D=V_2\cap D$ for some open set $V_1\subseteq M$ and some closed set $V_2\subseteq M$. Therefore, $q(V)=q(V_1)\cap C=q(V_2)\cap C$ is open and closed relative to $C$, because $q$ is a continuous and open map. As $x\in q(V)$ and $y\notin q(V)$ we conclude that $C$ is not connected, a contradiction.
\end{proof}

\begin{prop}\label{prop:semiconjcw} Let $M$ and $N$ compact metric spaces, $f\colon M\to M$ and $g\colon N\to N$ homeomorphisms and $q\colon M\to N$ a continuous and open onto map such that $q\circ f=g\circ q$, and with the property that $q^{-1}(y)$ is totally disconnected for all $y\in N$. If $f$ is cw\hyph expansive then $g$ is cw\hyph expansive.
\end{prop}
\begin{proof} Let $\expc>0$ be a cw\hyph expansivity constant for $f$. Suppose that $g$ is not cw\hyph expansive. Then for all $k\in\N$ there exist a non trivial continuum $C'_k\subseteq N$ such that
$\diam g^n(C'_k)<1/k$ for all $n\in\Z$. By Lemma \ref{lem:hayconexo}, for each $k\in\N$ there exists a non trivial connected component $D'_k\subseteq M$ of $q^{-1}(C'_k)$. Then, as $f$ is cw\hyph expansive, there exists $n_k\in\Z$ such that $\diam f^{n_k}(D'_k)\geq\expc$ for all $k\in\N$. Let $D_k=f^{n_k}(D'_k)$ and $C_k=g^{n_k}(C'_k)$ for $k\in\N$. We have that $\diam D_k\geq\expc$, $q(D_k)=C_k$ and $\diam C_k<1/k$, for all $k\in\N$. Taking a subsequence we have that $D_{k_m}\to D$ with respect to the Hausdorff metric, where $D$ is a continuum which will satisfy $\diam D\geq\expc$. Hence, as $q$ is continuous and $\diam C_{m_k}<1/m_k$ we see that $q(D)={y}$, for some $y\in N$. This is a contradiction since by hypothesis $q^{-1}(y)$ must be totally disconnected.
\end{proof}

Let $\Sdos$ be the 2-sphere obtained as the quotient of the 2-torus $\Tdos=\R^2/\Z^2$ by the map $T(x,y)=(-x,-y)$, and $q\colon \Tdos\to \Sdos$ the canonical map. Let $f\colon \Tdos\to \Tdos$ the linear Anosov diffeomorphism given by the matrix $\bigl[\begin{smallmatrix}2&1\\1&1\end{smallmatrix}\bigr]$, and $g\colon \Sdos\to \Sdos$ the induced homeomorphism.
More details of this construction can be found in \cite{Walters}*{Example 1, p. 140}.

\begin{thm}
\label{thm:ejemplocw+sh}
The homeomorphism $g\colon \Sdos\to \Sdos$ is cw\hyph expansive
and has the shadowing property.
\end{thm}

\begin{proof}
Observe that the canonical map $q\colon \Tdos\to \Sdos$ is open. Then, as $f$ has the shadowing property, by Proposition \ref{prop:semiconjsombreado}, $g$ has the shadowing property.
On the other hand, as $f$ is expansive, and $q$ is a finite-to-1 map (in fact, each point has at most two preimages), we can apply Proposition \ref{prop:semiconjcw} and conclude that $g$ is cw\hyph expansive.
\end{proof}

\begin{prob}For simplicity let us say that a homeomorphism is \emph{cw-Anosov} if it is cw\hyph expansive and has the shadowing property.
Besides the map $g$ of the sphere, there are other examples of cw-Anosov homeomorphisms, of course, Anosov diffeomorphisms.
It would be interesting to classify all the cw-Anosov homeomorphisms of compact surfaces.
Some natural questions arises: are there cw-Anosov homeomorphisms of the torus not being (conjugate to) Anosov?
Does the genus two surface admit cw-Anosov homeomorphisms?
Does local stable and unstable sets of a cw-Anosov homeomorphism define singular foliations? Are these local stable sets locally connected?
\end{prob}

\begin{rmk}
The cw\hyph expansivity of $g$ could have been deduced from \cite{ArDend}*{Proposition 2.2.1} where
it is proved that in fact $g$ is cw2\hyph expansive, which means that there is $\expc>0$ such that
if $C_1,C_2\subseteq \Sdos$ are continua such that
$\diam g^n(C_1)\leq\expc$ for all $n\geq 0$ and $\diam g^{-n}(C_2)\leq\expc$ for all
$n\geq 0$ then $C_1\cap C_2$ has at most two points.
It is clear that cw2\hyph expansivity implies cw\hyph expansivity.
However, we think that the proof of the cw\hyph expansivity of $g$ given in Theorem \ref{thm:ejemplocw+sh}
(based on Proposition \ref{prop:semiconjcw}) is simpler and clearer than the proof in \cite{ArDend}.
\end{rmk}

\vspace{20mm}

\noindent Mauricio Achigar,\\
{\tt machigar@unorte.edu.uy},\\
\\
\noindent Alfonso Artigue,\\
{\tt artigue@unorte.edu.uy},\\
\\
\noindent José Vieitez,\\
{\tt jvieitez@unorte.edu.uy},\\
\\
{\sc Departamento de Matemática y Estadística del Litoral},\\
{\sc Centro Universitario Regional Litoral Norte},\\
{\sc Universidad de la República.}\\
25 de Agosto 281, Salto (50000), Uruguay.

\end{document}